\documentclass[12pt,a4paper,reqno]{amsart}

\usepackage{amsmath, amssymb, amsthm, mathtools}
\usepackage[margin=2.5cm, headsep=0.75cm, footskip=0.75cm]{geometry}
\usepackage[hang, flushmargin]{footmisc}
\usepackage{enumitem}
\usepackage[english]{isodate}
\usepackage[dvipsnames]{xcolor}
\usepackage{hyperref}
\usepackage[noabbrev, capitalise, nameinlink, nosort]{cleveref}

\hypersetup{pdfauthor={Becky Armstrong, Kevin Aguyar Brix, Toke Meier Carlsen, and Søren Eilers}, pdftitle={Conjugacy of local homeomorphisms via groupoids and C*-algebras}, colorlinks=true, linkcolor=blue, linkbordercolor=blue, citecolor=red, citebordercolor=red, linktocpage=true}

\newcommand{\N}{\mathbb{N}}
\newcommand{\R}{\mathbb{R}}
\newcommand{\T}{\mathbb{T}}
\newcommand{\Z}{\mathbb{Z}}

\renewcommand{\AA}{\mathcal{A}}
\newcommand{\BB}{\mathcal{B}}
\newcommand{\D}{\mathcal{D}}
\newcommand{\G}{\mathcal{G}}
\newcommand{\LL}{\mathcal{L}}
\newcommand{\OO}{\mathcal{O}}
\newcommand{\h}{\mathcal{H}}

\newcommand{\X}{\mathsf{X}}
\newcommand{\Y}{\mathsf{Y}}

\DeclareMathOperator{\Iso}{Iso}

\DeclareMathOperator{\id}{id}
\DeclareMathOperator{\orb}{orb}
\DeclareMathOperator{\osupp}{osupp}
\DeclareMathOperator{\dom}{dom}

\DeclareMathOperator{\ran}{ran}

\newcommand{\restr}[1]{\ensuremath{\vert_{#1}}}
\newcommand{\domsigma}[1]{\dom(\sigma_{#1})}
\newcommand{\domsigmak}[2]{\dom(\sigma^{#2}_{#1})}

\newcommand{\definitionemph}[1]{\emph{\textcolor{magenta}{#1}}\index{#1}}

\let\tilde\widetilde
\let\le\leqslant{}
\let\ge\geqslant{}
\let\subset\subseteq{}
{}

\numberwithin{equation}{section}
\newtheorem{lemma}{Lemma}
\numberwithin{lemma}{section}
\newtheorem{corollary}[lemma]{Corollary}
\newtheorem{theorem}[lemma]{Theorem}
\newtheorem{proposition}[lemma]{Proposition}

\theoremstyle{definition}
\newtheorem{definition}[lemma]{Definition}
\newtheorem{example}[lemma]{Example}

\newtheorem{remark}[lemma]{Remark}
\newtheorem*{theorem*}{Theorem}

\makeatletter\g@addto@macro\bfseries{\boldmath}\makeatother

\makeatletter
\let\origsection\section
\renewcommand\section{\@ifstar{\starsection}{\nostarsection}}
\newcommand\sectionspace{\vspace{0.5ex}}
\newcommand\nostarsection[1]{\sectionspace\origsection{#1}\sectionspace}
\newcommand\starsection[1]{\sectionspace\origsection*{#1}\sectionspace}
\makeatother

\setenumerate{listparindent=\parindent}
\setlist[enumerate]{font=\normalfont}

\crefname{enumi}{}{}
\crefname{enumi}{}{}

\crefname{equation}{equation}{equations}
\crefname{condition}{condition}{conditions}

\subjclass[2020]{37A55 (primary), 37B10, 46L55 (secondary).}

\date{\today}

\begin{document}

\title[Conjugacy of local homeomorphisms via groupoids and C*-algebras]{Conjugacy of local homeomorphisms via groupoids and C*-algebras}

\author[B.~Armstrong]{Becky Armstrong}
\address[B.~Armstrong]{Mathematical Institute, WWU M\"unster, Einsteinstr.\ 62, 48149 M\"unster, Germany}
\email{\href{mailto:becky.armstrong@uni-muenster.de}{becky.armstrong@uni-muenster.de}}

\author[K.A.~Brix]{Kevin Aguyar Brix}
\address[K.A.~Brix]{School of Mathematics and Statistics, University of Glasgow, University Place, Glasgow G12 8QQ, United Kingdom}
\email{\href{mailto:kabrix.math@fastmail.com}{kabrix.math@fastmail.com}}

\author[T.M.~Carlsen]{Toke Meier Carlsen}
\address[T.M.~Carlsen]{Department of Science and Technology, University of the Faroe Islands, Vestara Bryggja 15, FO-100 T\'orshavn, the Faroe Islands}
\email{\href{mailto:toke.carlsen@gmail.com}{toke.carlsen@gmail.com}}

\author[S.~Eilers]{S{\o}ren Eilers}
\address[S.~Eilers]{Department of Mathematical Sciences, University of Copenhagen, Universitetsparken 5, DK-2100 Copenhagen, Denmark}
\email{\href{mailto:eilers@math.ku.dk}{eilers@math.ku.dk}}

\keywords{Conjugacy, local homeomorphism, Deaconu--Renault system, groupoid, C*-algebra}

\begin{abstract}
We investigate dynamical systems consisting of a locally compact Hausdorff space equipped with a partially defined local homeomorphism. Important examples of such systems include self-covering maps, one-sided shifts of finite type and, more generally, the boundary-path spaces of directed and topological graphs. We characterise topological conjugacy of these systems in terms of isomorphisms of their associated groupoids and C*-algebras. This significantly generalises recent work of Matsumoto and of the second- and third-named authors.
\end{abstract}

\maketitle

\section{Introduction}
The tradition of constructing operator algebras from dynamical systems originated with the seminal work of Murray and von Neumann on the group von Neumann algebra construction \cite{MvN}. This approach has not only produced a plethora of interesting examples of operator algebras, but has also led to interesting results pertaining to topological dynamics. Prime examples of such results are those of Giordano, Putnam, and Skau \cite[Theorems~2.1 and 2.2]{GPS95} that use C*-crossed products to show that Cantor minimal systems can be classified up to (strong) orbit equivalence by K-theory.

C*-algebras constructed from dynamical systems have mainly been obtained as C*-crossed products of actions of locally compact groups on topological spaces. However, in the last forty years or so, C*-algebras constructed from non-invertible actions on topological spaces have also attracted considerable attention. This approach inspired Krieger's dimension group \cite{Krieger80a, Krieger80b}, which has had a tremendous impact on the study of symbolic dynamical systems. A particularly interesting class of these kinds of C*-algebras is the class of C*-algebras arising from a dynamical system consisting of a local homeomorphism acting on a locally compact Hausdorff space \cite{D95, Renault2000}. These C*-algebras come equipped with a family of symmetries induced by functions on the space, and by using a topological groupoid approach, we show how to recover the conjugacy class of the underlying system from the C*-algebra and its family of symmetries.

A system consisting of a locally compact Hausdorff space $\X$ together with a local homeomorphism $\sigma_\X$ between open subsets of $\X$ is called a \definitionemph{Deaconu--Renault system} (see, for example, \cite{D95, Renault2000, CRST, ABS}). Examples of Deaconu--Renault systems include self-covering maps \cite{D95, EV06}, one-sided shifts of finite type~\cite{Williams1973, dlbm:isdc, Kitchens}, the boundary-path space of a directed graph together with the shift map \cite{Webster2014, BCW17}, and more generally, the boundary-path space of a topological graph together with the shift map \cite{KL2017}, the one-sided edge shift space of an ultragraph together with the restriction of the shift map to points with nonzero length \cite{GR19}, the full one-sided shift over an infinite alphabet together with the restriction of the shift map to points with nonzero length \cite{OMW14}, the cover of a one-sided shift space constructed in \cite{BC20b}, and more generally, the canonical local homeomorphism extension of a locally injective map constructed in \cite{Thomsen2011}.

A C*-algebra is naturally associated to a Deaconu--Renault system via a groupoid construction (see, for example, \cite{D95, Renault2000, CRST}), and the class of such C*-algebras includes crossed products by actions of $\Z$ on locally compact Hausdorff spaces, Cuntz--Krieger algebras~\cite{CK80}, graph C*-algebras~\cite{Raeburn2005}, and, via Katsura's topological graphs \cite{Katsura2004}, all Kirchberg algebras (i.e.~all purely infinite, simple, nuclear, separable C*-algebras) satisfying the universal coefficient theorem (see \cite{Katsura2008}), C*-algebras associated with one-sided shift spaces \cite{BC20b}, and C*-algebras of locally injective surjective maps \cite{Thomsen2011}.

It is natural to ask how much information from a dynamical system can be recovered by C*-algebraic data associated with it. It is known that an action of a locally compact group on a topological space can be recovered up to conjugacy from its C*-crossed product together with the corresponding dual action (see \cite[Proposition~4.3]{KOQ18}), and similar results have been obtained in more specialised settings; see, for example, \cite[Theorem~2.4]{GPS95}, \cite[Theorem~3.6]{BT98}, \cite[Theorem~1.2]{Li18}, and \cite[Corollary~7.5 and Theorem~9.1]{CRST}.

In recent years, a similar approach has been used to encode dynamical relations between irreducible shifts of finite type into Cuntz--Krieger algebras.
Matsumoto has been a driving force in this endeavour with his characterisations of continuous orbit equivalence~\cite{Matsumoto2010} and one-sided eventual conjugacy~\cite{Matsumoto2017}, among others. Together with Matui, Matsumoto characterised flow equivalence as diagonal-preserving $*$-isomorphism of stabilised Cuntz--Krieger algebras using groupoids~\cite{MM14} (see also~\cite{CEOR}), and this led the third-named author and Rout to prove similar characterisations for two-sided conjugacy of shifts of finite type~\cite{Carlsen-Rout2017}. Complementing the work of two of the authors~\cite{BC20a}, Matsumoto recently proved that a one-sided conjugacy of irreducible shifts of finite type can be encoded into Cuntz--Krieger algebras using gauge actions~\cite{Matsumoto2021a}. Matsumoto's methods immediately inspired us and helped shape the present work. Since then, Matsumoto has released three other works on related topics~\cite{Matsumoto2020, Matsumoto2021b, Matsumoto2021c}.

In this paper we provide a characterisation of conjugacy of a pair of Deaconu--Renault systems in terms of isomorphisms of their groupoids and their C*-algebras. Our results are summarised in \cref{cor:summary}. Since we work in the general framework of Deaconu--Renault systems, our work complements (and applies) the groupoid reconstruction theory of~\cite{CRST}, which is based on the pioneering work of Renault~\cite{Renault1980, Renault2008} and Kumjian~\cite{Kumjian}. We prove in \cref{prop:gauge-intertwining} that a $*$-isomorphism of the C*-algebras of Deaconu--Renault groupoids that intertwines a sufficiently rich collection of automorphisms induces a conjugacy between the underlying systems. It is noteworthy that we do not require the $*$-isomorphism to be diagonal-preserving. By restricting to the case of one-sided shifts of finite type, we therefore not only recover, but strengthen Matsumoto's \cite[Theorem~1]{Matsumoto2021a} characterisation of one-sided conjugacy (see \cref{cor:graphs}). The proof requires a technical result (\cref{lem:technical-lemma}) which relates actions of the C*-algebra to cocycles on the groupoid, and we believe that this may be of independent interest.

In future work~\cite{ABCE}, we shall approach conjugacy of directed graphs from an algorithmic and combinatorial point of view related to~\cite{ER}.

This paper is organised as follows. In~\cref{sec:prelim} we introduce the necessary notation and preliminaries and establish some basic facts about Deaconu--Renault systems, conjugacy of Deaconu--Renault systems, and Deaconu--Renault groupoids and their C*-algebras. We also provide several examples of Deaconu--Renault systems and Deaconu--Renault groupoids and their C*-algebras and relate them to previous work. In \cref{sec:characterisation} we prove our main results relating conjugacy of Deaconu--Renault systems to the associated groupoids and C*-algebras. Our results are summarised in \cref{cor:summary}, which follows from \cref{prop:sigma,prop:groupoid-conjugacy,prop:gauge-intertwining}.

\section*{Acknowledgements}
The first-named author would like to thank the other three authors for their hospitality during a visit to the University of Copenhagen in 2019. This visit was funded by an Australian Mathematical Society Lift-Off Fellowship. The second-named author is supported by the Carlsberg Foundation through an Internationalisation Fellowship, and is grateful to the people at the University of Wollongong for providing a hospitable and stimulating environment during his stay. The third-named author is supported by Research Council Faroe Islands, and is grateful to the Department of Mathematical Sciences at the University of Copenhagen for their hospitality during a visit in 2019. The fourth-named author was supported by the DFF-Research Project 2 ``Automorphisms and Invariants of Operator Algebras'', no.~7014-00145.

\section{Preliminaries} \label{sec:prelim}
Here we introduce the basic concepts and notation used throughout the paper. We denote the real numbers by $\R$, the integers by $\Z$, the nonnegative integers by $\N$, and the positive integers by $\N_+$. We denote the cardinality of a set $A$ by $\lvert A \rvert$. Given a locally compact Hausdorff space $\X$, we write $C_b(\X)$ for the set of continuous bounded functions from $\X$ to the complex numbers, we write $C_0(\X)$ for the subset of $C_b(\X)$ consisting of functions that vanish at infinity, and we write $C_c(\X)$ for the subset of $C_0(\X)$ consisting of functions that have compact support. The spaces $C_0(\X)$ and $C_b(\X)$ are abelian C*-algebras, and $C_b(\X)$ is (isomorphic to) the multiplier algebra of $C_0(\X)$ (see, for instance, \cite[Example~3.1.3]{Murphy}).

\subsection{Deaconu--Renault systems}
A \definitionemph{Deaconu--Renault system} is a pair $(\X,\sigma_\X)$ consisting of a locally compact Hausdorff space $\X$ and a partially defined local homeomorphism $\sigma_\X\colon \domsigma{\X}\to \ran(\sigma_\X)$, where both $\domsigma{\X}$ and $\ran(\sigma_\X)$ are open subsets of $\X$. Let $\sigma_\X^0 \coloneqq \id_\X$, and inductively define $\dom(\sigma_\X^k) \coloneqq \sigma_\X^{-1}(\dom(\sigma_\X^{k - 1}))$ for $k \in \N$. Then for every $k \in \N$, the map $\sigma_\X^k\colon \dom(\sigma_\X^k) \to \ran(\sigma_\X^k)$ defined by $\sigma_\X^k(x) \coloneqq \sigma_\X^{k-1}(\sigma_\X(x))$ is a local homeomorphism onto an open subset of $\X$. Whenever we write $\sigma_\X^k(x)$ it is to be understood that $x \in \dom(\sigma_\X^k)$. The \definitionemph{orbit} of a point $x \in \X$ is the subset
\[
\orb_\X(x) \coloneqq \bigcup_{k, l \in \N} \sigma_\X^{-l} \big( \sigma_\X^k(x) \big),
\]
and a pair of points $x, y \in \X$ are in the same orbit if and only if $\sigma_\X^k(x) = \sigma_\X^l(y)$ for some $k, l \in \N$. A point $x \in \X$ is \definitionemph{periodic} with \definitionemph{period} $p \in \N_+$ if $x = \sigma_\X^p(x)$, it is \definitionemph{eventually periodic} if $\sigma_\X^n(x)$ is periodic for some $n\in \N$, and it is \definitionemph{aperiodic} if it is not eventually periodic. We say that a Deaconu--Renault system $(\X,\sigma_\X)$ is \definitionemph{topologically free} if the set $\{ x \in \X : x \text{ is not periodic} \}$ is dense in $\X$, and that the system is \definitionemph{second-countable} if $\X$ is second-countable.

Similar systems were studied independently by Deaconu in~\cite{D95} and by Renault in~\cite{Renault2000} (as singly generated dynamical systems). Here we follow the terminology of~\cite[Section~8]{CRST}.

\begin{definition}
Let $(\X,\sigma_\X)$ and $(\Y,\sigma_\Y)$ be Deaconu--Renault systems.
We call a homeomorphism $h\colon \X \to \Y$ a \definitionemph{conjugacy} if $h(\sigma_\X(x)) = \sigma_\Y(h(x))$ and $h^{-1}(\sigma_\Y(y)) = \sigma_\X(h^{-1}(y))$ for all $x \in \domsigma{\X}$ and $y \in \domsigma{\Y}$. We say that the systems $(\X,\sigma_\X)$ and $(\Y,\sigma_\Y)$ are \definitionemph{conjugate} if there exists a conjugacy $h\colon \X \to \Y$.
\end{definition}

\begin{example}
Let $E$ be a directed graph, let $\partial E$ be the boundary-path space of $E$ introduced in \cite{Webster2014}, and let $\sigma_E\colon \partial E^{\ge 1}\to\partial E$ be the shift map described in \cite[Section~2.2]{BCW17}. Then $(\partial E,\sigma_E)$ is a Deaconu--Renault system (see \cite[Section~2.2]{BCW17}). In \cite[Section~6.1]{BCW17}, two directed graphs $E$ and $F$ are defined to be conjugate precisely when the corresponding Deaconu--Renault systems $(\partial E,\sigma_E)$ and $(\partial F,\sigma_F)$ are conjugate; cf.~\cref{lem:alternative-conjugacy}.
\end{example}

\begin{example}
Let $\AA$ be an infinite countable set, let $(\Sigma_{\AA},\sigma)$ be the one-sided full shift over $\AA$ defined in \cite[Definitions~2.1 and~2.22]{OMW14}, and let $\sigma_{\Sigma_{\AA}}$ be the restriction of $\sigma$ to $\Sigma_{\AA}\setminus\{\vec{0}\}$. It follows from \cite[Proposition~2.5 and Proposition~2.23]{OMW14} that $(\Sigma_{\AA},\sigma_{\Sigma_{\AA}})$ is a Deaconu--Renault system.

If $\AA$ and $\BB$ are infinite countable sets and $\phi\colon\Sigma_{\AA}\to \Sigma_{\BB}$ is a conjugacy as defined in \cite[Definition~4.8]{OMW14}, then it follows from \cite[Proposition~4.2 and Remark~4.9]{OMW14} that $\phi$ is also a conjugacy between the Deaconu--Renault systems $(\Sigma_{\AA},\sigma_{\Sigma_{\AA}})$ and $(\Sigma_{\BB},\sigma_{\Sigma_{\BB}})$.
\end{example}

\begin{example}
Let $\G$ be an ultragraph satisfying the condition (RFUM) introduced in \cite{GR19}, let $(\X_\G,\sigma_\G)$ be the one-sided edge shift of $\G$ constructed in \cite{GR19}, and let $\sigma_{\X_\G}$ be the restriction of $\sigma_\G$ to $\X_\G\setminus \mathfrak{p}^0$, where $\mathfrak{p}^0$ is the set of ultrapaths of length $0$ (see \cite[Section~2.1]{GR19}). It follows from \cite[Propositions~3.7, 3.12, and~3.16]{GR19} that $(\X_\G,\sigma_{\X_\G})$ is a Deaconu--Renault system.

If $\G_1$ and $\G_2$ are two ultragraphs satisfying condition (RFUM), then a map $\phi\colon \X_{\G_1} \to \X_{\G_2}$ is a conjugacy between the Deaconu--Renault systems $(\X_{\G_1},\sigma_{\X_{\G_1}})$ and $(\X_{\G_2},\sigma_{\X_{\G_2}})$ if and only if it is a length-preserving conjugacy, as defined in \cite[Definition~3.19]{GR19}.
\end{example}

The following is an example of a homeomorphism $h\colon \X \to \Y$ between two Deaconu--Renault systems $(\X,\sigma_\X)$ and $(\Y,\sigma_\Y)$ that is not a conjugacy, even though $h(\sigma_\X(x)) = \sigma_\Y(h(x))$ for all $x \in \domsigma{\X}$.

\begin{example} \label{eg:complex-circle}
Consider the graph $E$ consisting of a single vertex, and the graph $F$ consisting of a single loop. The boundary-path spaces $\partial E$ and $\partial F$ are both singletons, and so there is a homeomorphism $h\colon \partial E \to \partial F$ that trivially satisfies $h(\sigma_E(x)) = \sigma_F(h(x))$ for all $x \in \dom(\sigma_E)$ (since $\dom(\sigma_E) = \varnothing$). However, $(\partial E, \sigma_E)$ and $(\partial F, \sigma_F)$ are not conjugate systems, because $\dom(\sigma_F) = \partial F$, whereas $\dom(\sigma_E) = \varnothing$.
\end{example}

It will be convenient later to have slight reformulations of the conjugacy condition at our disposal. Note that the conditions in item~\cref{item:inverse-images} below are set equalities.

\begin{lemma} \label{lem:alternative-conjugacy}
Let $(\X,\sigma_\X)$ and $(\Y,\sigma_\Y)$ be Deaconu--Renault systems, and let $h\colon \X \to \Y$ be a homeomorphism. The following statements are equivalent.
\begin{enumerate}[label=(\arabic*)]
\item \label{item:conjugacy} $h\colon \X \to \Y$ is a conjugacy.
\item \label{item:domain} $h(\domsigma{\X}) = \domsigma{\Y}$, and $h \circ \sigma_\X = \sigma_\Y \circ h$ on $\domsigma{\X}$.
\item \label{item:inverse-images} $h(\sigma_\X^{-1}(x)) = \sigma_\Y^{-1}(h(x))$ and $h^{-1}(\sigma_\Y^{-1}(y)) = \sigma_\X^{-1}(h^{-1}(y))$, for all $x \in \X$ and $y \in \Y$.
\end{enumerate}
In particular, if the two systems have globally defined dynamics (i.e.~if $\domsigma{\X} = \X$ and $\domsigma{\Y} = \Y$), then the condition $h \circ \sigma_\X = \sigma_\Y \circ h$ is equivalent to $h$ being a conjugacy.
\end{lemma}

\begin{proof}
\cref{item:conjugacy}$\iff$\cref{item:domain}:
Assume first that $h\colon \X \to \Y$ is a conjugacy. If $x \in \domsigma{\X}$, then $h(\sigma_\X(x)) = \sigma_\Y(h(x))$, and so $h(\domsigma{\X}) \subset \domsigma{\Y}$. For the reverse inclusion, fix $y \in \domsigma{\Y}$. Then $\sigma_\X(h^{-1}(y)) = h^{-1}(\sigma_\Y(y))$, and so $y \in h(\domsigma{\X})$.

For the converse, we need to verify that $h^{-1} \circ \sigma_\Y = \sigma_\X \circ h^{-1}$ on $\domsigma{\Y}$. Fix $y \in \domsigma{\Y}$, and let $x \coloneqq h^{-1}(y) \in \domsigma{\X}$. Then $h(\sigma_\X(x)) = \sigma_\Y(h(x)) = \sigma_\Y(y)$, and so $h^{-1}(\sigma_\Y(y)) = \sigma_\X(x) = \sigma_\X(h^{-1}(y))$, as required.

\cref{item:conjugacy}$\iff$\cref{item:inverse-images}:
Assume first that $h$ is a conjugacy, and fix $x \in \X$. Suppose that $\sigma_\X^{-1}(x)$ is nonempty, and fix $z \in \sigma_\X^{-1}(x)$. Then $\sigma_\Y(h(z)) = h(\sigma_\X(z)) = h(x)$, and thus $h(z) \in \sigma_\Y^{-1}(h(x))$. It follows that $h(\sigma_\X^{-1}(x)) \subset \sigma_\Y^{-1}(h(x))$. For the reverse inclusion, fix $w \in \sigma_\Y^{-1}(h(x))$. Then $\sigma_\X(h^{-1}(w)) = h^{-1}(\sigma_\Y(w)) = x$, and so $w \in h(\sigma_\X^{-1}(x))$. Therefore, $\sigma_\Y^{-1}(h(x)) \subset h(\sigma_\X^{-1}(x))$. Suppose instead that $\sigma_\X^{-1}(x)$ is empty. We claim that $\sigma_\Y^{-1}(h(x))$ is also empty. Suppose for contradiction that there exists $w \in \sigma_\Y^{-1}(h(x))$. Then $\sigma_\X(h^{-1}(w)) = h^{-1}(\sigma_\Y(w)) = x$, and thus $h^{-1}(w) \in \sigma_\X^{-1}(x)$, which contradicts the hypothesis that $\sigma_\X^{-1}(x) = \varnothing$. A similar argument shows that if $h$ is a conjugacy, then $h^{-1}(\sigma_\Y^{-1}(y)) = \sigma_\X^{-1}(h^{-1}(y))$ for all $y \in \Y$.

For the converse, fix $x \in \domsigma{\X}$, and let $w \coloneqq h(\sigma_\X(x))$.Then
\[
x \in \sigma_\X^{-1}(\sigma_\X(x)) = \sigma_\X^{-1}(h^{-1}(w)) = h^{-1}(\sigma_\Y^{-1}(w))
\]
by hypothesis, and so $\sigma_\Y(h(x)) = w = h(\sigma_\X(x))$. A similar argument shows that $h^{-1}(\sigma_\Y(y)) = \sigma_\X(h^{-1}(y))$ for all $y \in \domsigma{\Y}$, and so $h$ is a conjugacy.
\end{proof}

\begin{remark}
As~\cref{eg:complex-circle} shows, there are some subtleties involved in the definition of a conjugacy of arbitrary Deaconu--Renault systems: it is important that we have conditions on both $h$ and its inverse $h^{-1}$. However, if the two systems have globally defined dynamics (i.e.~if $\domsigma{\X} = \X$ and $\domsigma{\Y} = \Y$), then \cref{lem:alternative-conjugacy} implies that the condition $h \circ \sigma_\X = \sigma_\Y \circ h$ is sufficient. In the case of directed graphs, the condition of having globally defined dynamics is equivalent to there being no singular vertices.
\end{remark}

We now introduce two maps $(\sigma_\X)^*\colon C_0(\X) \to C_b(\domsigma{\X})$ and $(\sigma_\X)_*\colon C_c(\domsigma{\X}) \to C_c(\X)$ that we will use in \cref{cor:summary,prop:sigma} to characterise conjugacy of Deaconu--Renault systems.

Suppose that $(\X,\sigma_\X)$ is a Deaconu--Renault system. We define the map $(\sigma_\X)^*\colon C_0(\X) \to C_b(\domsigma{\X})$ by
\[
(\sigma_\X)^*(f)(x) \coloneqq f(\sigma_\X(x)),
\]
for all $f \in C_0(\X)$ and $x \in \domsigma{\X}$. The following example shows that the range of $(\sigma_\X)^*$ is in general larger than the collection of functions vanishing at infinity.

\begin{example}
Let $E$ be the graph with two vertices $v$ and $w$ such that $v$ emits infinitely many edges $\{e_n : n \in \N\}$ to $w$. Then $1_{Z(w)} \in C_0(\partial E)$, and $1_{Z(w)} \circ \sigma_E$ is defined and nonzero (and bounded) on the entire non-compact set $\dom(\sigma_E) = \{ e_n : n \in \N \}$.
\end{example}

Since $\sigma_\X$ is a local homeomorphism, there is a map $(\sigma_\X)_*\colon C_c(\domsigma{\X}) \to C_c(\X)$ given by
\[
(\sigma_\X)_*(f)(x) \coloneqq \sum_{z \in \sigma_\X^{-1}(x)} f(z),
\]
for all $f \in C_c(\domsigma{\X})$ and $x \in \X$.

\subsection{Deaconu--Renault groupoids}

Every Deaconu--Renault system $(\X,\sigma_\X)$ gives rise to a \definitionemph{Deaconu--Renault groupoid}
\[
\G_\X \coloneqq \bigcup_{m,n \in \N} \big\{ (x, m - n, y) \in \domsigmak{\X}{m} \times \{m - n\} \times \domsigmak{\X}{n} \,:\, \sigma_\X^m(x) = \sigma_\X^n(y) \big\},
\]
with composable pairs $\G_\X^{(2)} \coloneqq \big\{ \big((x,p,y), (w,q,z)\big) \in \G_\X \times \G_\X : y = w \big\}$ (cf.~\cite[Section~8]{CRST}, \cite[Definition~5.4]{KL2017}, and \cite[Definition~2.4]{Renault2000}). Multiplication is defined on $\G_\X^{(2)}$ by $(x,p,y)(y,q,z) \coloneqq (x,p+q,z)$, and inversion is defined on $\G_\X$ by $(x, p, y)^{-1} \coloneqq (y, -p, x)$. The range and source maps of $\G_\X$ are $r\colon (x, p, y) \mapsto x$ and $s\colon (x, p, y) \mapsto y$, respectively. The unit space of $\G_\X$ is $\G_\X^{(0)} = \{ (x, 0, x) \in \G_\X : x \in \X\}$, and we identify it with $\X$ via the map $(x, 0, x) \mapsto x$. Given $m, n \in \N$ and open subsets $U$ and $V$ of $\X$, we define
\begin{equation} \label{eq:basic-open-groupoid}
Z(U, m, n, V) \coloneqq \big\{ (x, m - n, y) \in \G_\X \,:\, x \in U, \, y \in V, \, \sigma_\X^m(x) = \sigma_\X^n(y) \big\},
\end{equation}
and the collection of all such sets forms a basis for a topology on $\G_\X$. Equivalently, the topology is generated by sets of the form~\labelcref{eq:basic-open-groupoid}, where, in addition, $\sigma_\X^m\restr{U}$ and $\sigma_\X^n\restr{V}$ are homeomorphisms onto their images, and $\sigma_\X^m(U) = \sigma_\X^n(V)$. Under this topology, $\G_\X$ is an amenable locally compact Hausdorff groupoid which is \definitionemph{\'etale}, in the sense that the range and source maps are local homeomorphisms. A set $B \subset \G_\X$ is called a \definitionemph{bisection} of $\G_\X$ if there is an open subset $U$ of $\G_\X$ such that $B \subset U$, and $r\restr{U}$ and $s\restr{U}$ are homeomorphisms onto open subsets of $\G_\X^{(0)}$. Every \'etale groupoid has a basis consisting of open bisections, and by choosing $U, V \subset \X$ sufficiently small, the sets of the form~\labelcref{eq:basic-open-groupoid} become open bisections of $\G_\X$. Moreover, $\G_\X$ is second-countable when $\X$ is second-countable, and $\G_\X$ is \definitionemph{ample} (meaning it has a basis of \emph{compact} open bisections) when $\X$ is totally disconnected. Given Deaconu--Renault groupoids $\G_\X$ and $\G_\Y$ and a map $\psi\colon \G_\X \to \G_\Y$, we write $\psi^{(0)}\colon \X \to \Y$ for the map induced by the restriction $\psi\restr{\G_\X^{(0)}}\colon \G_\X^{(0)} \to \G_\Y^{(0)}$.

The \definitionemph{isotropy} of a groupoid $\G$ is the subgroupoid $\Iso(\G) \coloneqq \{ \gamma \in \G : r(\gamma) = s(\gamma) \}$. If $\G$ is a second-countable locally compact Hausdorff \'etale groupoid, then so is the interior $\Iso(\G)^\circ$ of the isotropy of $\G$. Since the unit space of an \'etale groupoid is open, we have $\G^{(0)} \subseteq \Iso(\G)^\circ$ if $\G$ is \'etale.

We now prove a well known result that we were unable to find in the literature.

\begin{lemma} \label{lem:top-free-implies-effective}
If $(\X,\sigma_\X)$ is a topologically free Deaconu--Renault system, then $\Iso(\G_\X)^\circ = \X$.
\end{lemma}

\begin{proof}
Suppose for contradiction that there exist $x \in \X$ and $k,l \in \N$ with $k \ne l$ and $\sigma_\X^k(x)=\sigma_\X^l(x)$ such that $(x,k-l,x)\in \Iso(\G_\X)^\circ$. Then there are nonempty open subsets $U,V \subseteq \X$ such that $\sigma_\X^k\restr{U}$ and $\sigma_\X^l\restr{V}$ are homeomorphisms onto their images, $\sigma_\X^k(U) = \sigma_\X^l(V)$, and $Z(U,k,l,V) \subseteq \Iso(\G_\X)$. We therefore have that $\sigma_\X^k(x') = \sigma_\X^l(x')$ for any $x' \in U$. It follows that any element of the open subset $\sigma_\X^k(U)$ is periodic with period $\lvert k-l \rvert$, but this contradicts the assumption that $(\X,\sigma_\X)$ is topologically free. Thus, $\Iso(\G_\X)^\circ = \X$.
\end{proof}

Given a group $\Gamma$ and a function $f\colon \X \to \Gamma$, for each $k \in \N$ and $x \in \domsigmak{\X}{k}$, we write
\[
f^{(k)}(x) \coloneqq \sum_{i=0}^{k-1} f(\sigma_\X^i(x)),
\]
where it is understood that $f^{(0)} = 0$. We use additive notation because in all of our examples of interest $\Gamma$ will be abelian, but of course a similar expression using multiplicative notation makes sense.

We call a continuous groupoid homomorphism from $\G_\X$ into a topological group $\Gamma$ a \definitionemph{continuous cocycle}. Any continuous map $f\colon \X \to \Gamma$ gives rise to a continuous cocycle $c_f\colon \G_\X \to \Gamma$ defined by
\begin{equation} \label{eq:cocycle-induced-by-function}
c_f(x, m - n, y) \coloneqq f^{(m)}(x) - f^{(n)}(y) = \sum_{i=0}^{m-1} f(\sigma_\X^i(x)) - \sum_{j=0}^{n-1} f(\sigma_\X^j(y)),
\end{equation}
for $(x, m - n, y) \in \G_\X$ satisfying $\sigma_\X^m(x) = \sigma_\X^n(y)$. When $\Gamma = \Z$ and $f \equiv 1$, we obtain the \definitionemph{canonical continuous cocycle} $c_f \colon (x,p,y) \mapsto p$, which we denote by $c_\X$.

\subsection{Deaconu--Renault C*-algebras}
Given a locally compact Hausdorff \'etale groupoid $\G$, there are associated full and reduced groupoid C*-algebras $C^*(\G)$ and $C_r^*(\G)$ encoding the structure of $\G$; see, for instance, \cite{Sims-notes, Renault1980} for details. An argument similar to the one used in \cite[Lemma~3.5]{Sims-Williams} shows that Deaconu--Renault groupoids are amenable, so the two C*-algebras $C^*(\G_\X)$ and $C_r^*(\G_\X)$ are canonically $*$-isomorphic, and we shall not distinguish between them: we let $C^*(\G_\X)$ denote \emph{the} C*-algebra associated to $\G_\X$. Since $\G_\X$ is \'etale, the unit space $\G_\X^{(0)} \cong \X$ is open in $\G_\X$, and we view the \definitionemph{diagonal subalgebra} $C_0(\X)$ as a subalgebra of $C^*(\G_\X)$. Note that $C_0(\X)$ need not be a C*-diagonal (in the sense of Kumjian~\cite{Kumjian}), nor a Cartan subalgebra (in the sense of Renault~\cite{Renault2008}).

The \definitionemph{Pontryagin dual} of a locally compact abelian group $\Gamma$ is the locally compact group $\widehat{\Gamma}$ consisting of continuous group homomorphisms from $\Gamma$ to the circle group $\T$, endowed with the compact-open topology. A continuous cocycle $c\colon \G_\X \to \Gamma$ induces an action $\gamma^c \colon \widehat{\Gamma} \curvearrowright C^*(\G_\X)$ satisfying
\begin{equation} \label{eq:action-induced-by-cocycle}
\gamma^c_\chi(\xi)(\eta) = \chi(c(\eta)) \, \xi(\eta),
\end{equation}
for $\chi \in \widehat{\Gamma}$, $\xi \in C_c(\G)$, and $\eta \in \G$; cf.~\cite[Chapter~2, Section~5]{Renault1980}. In particular, there is a \definitionemph{weighted action} $\gamma^{\X,f}\colon \widehat{\Gamma} \curvearrowright C^*(\G_\X)$ associated to each $f \in C(\X,\Gamma)$, satisfying
\begin{equation}
\gamma_\chi^{\X, f}(\xi)(x,m - n, y) = \chi\big(f^{(m)}(x) - f^{(n)}(y)\big) \, \xi(x, m - n, y),
\end{equation}
for $\chi \in \widehat{\Gamma}$, $\xi \in C_c(\G_\X)$, and $(x,m - n, y) \in \G_\X$ satisfying $\sigma_\X^m(x) = \sigma_\X^n(y)$. The \definitionemph{canonical gauge action} $\gamma^\X\colon \T\curvearrowright C^*(\G_\X)$ is induced by the canonical continuous cocycle $c_\X$ on $\G_\X$.

\begin{example}
Let $\X$ be a locally compact Hausdorff space, and let $\sigma\colon \X \to \X$ be a homeomorphism. Then $(\X,\sigma)$ is a Deaconu--Renault system. The Deaconu--Renault groupoid $\G_\X$ of $(\X,\sigma)$ is isomorphic to the transformation groupoid $\X\times_\sigma\Z$ (see, for instance, \cite[Examples~I.1.2(a)]{Renault1980}), and $C^*(\G_\X)$ is isomorphic to the crossed product $C_0(\X) \rtimes_\sigma \Z$ by an isomorphism that restricts to the identity map on $C_0(\X)$ and intertwines the canonical gauge action of $C^*(\G_\X)$ and the dual action of $C_0(\X) \rtimes_\sigma \Z$ (see, for instance, \cite[Example~9.2.6]{Sims-notes}).
\end{example}

\begin{example}
If $\X$ is a locally compact Hausdorff space and $\sigma\colon \X \to \X$ is a covering map (that is, $\sigma$ is continuous and surjective, and for each $x \in \X$, there is an open neighbourhood $V$ of $x$ such that $T^{-1}(V)$ is a disjoint union of open sets $(U_i)_{i \in I}$ such that $\sigma\restr{U_i}$ is a homeomorphism onto $V$ for each $i \in I$), then $(\X,\sigma)$ is a Deaconu--Renault system. The Deaconu--Renault groupoid $\G_\X$ of $(\X,\sigma)$ is the groupoid $\Gamma$ considered in \cite{D95}. If $\X$ is compact and each $x \in \X$ has exactly $p$ preimages under $\sigma$, then according to \cite[Theorem~9.1]{EV06}, the C*-algebra $C^*(\G_\X)$ is isomorphic to the crossed product $C(\X)\rtimes_{\alpha,\LL}\N$, where $\alpha$ is the $*$-homomorphism from $C(\X)$ to $C(\X)$ given by $\alpha(f) \coloneqq f\circ\sigma$, and $\LL$ is the transfer operator from $C(\X)$ to $C(\X)$ given by $\LL(f)(x) \coloneqq \frac{1}{p} \sum_{y \in \sigma^{-1}(x)} f(y)$. The proof of \cite[Theorem~9.1]{EV06} goes through with minor modifications if the assumption that each $x \in \X$ has exactly $p$ preimages under $\sigma$ is dropped and the definition of $\LL(f)(x)$ is changed to $\frac{1}{\lvert \sigma^{-1}(x) \rvert} \sum_{y \in \sigma^{-1}(x)} f(y)$.
\end{example}

\begin{example}
Let $E$ be a directed graph. The Deaconu--Renault groupoid $\G_{\partial E}$ of the Deaconu--Renault system $(\partial E,\sigma_E)$ is the graph groupoid $\G_E$ described, among other places, in \cite[Section~2.3]{BCW17}, and $C^*(\G_{\partial E})$ is isomorphic to the graph C*-algebra $C^*(E)$ of $E$ by an isomorphism that maps $C_0(\partial E)$ onto the diagonal C*-subalgebra $\D(E)$ of $C^*(E)$ and intertwines the canonical gauge action $\gamma^{\partial E}\colon \T\curvearrowright C^*(\G_{\partial G})$ and the gauge action $\gamma^E\colon \T \curvearrowright C^*(E)$ (see, for instance, \cite[Proposition~2.2]{BCW17}).

In~\cite[Section~3]{Carlsen-Rout2017} the third-named author and Rout consider generalised gauge actions. A function $k\colon E^1\to \R$ extends to a function $k\colon E^*\to \R$ by setting $k\restr{E^0} \equiv 0$ and $k(e_1 \dotsb e_n) \coloneqq k(e_1) + \dotsb + k(e_n)$ for $e_1 \dotsb e_n \in E^*\setminus E^0$. There is a continuous cocycle $c_k\colon \G_E\to \R$ given by $c_k(\mu x, \lvert \mu \rvert - \lvert \nu \rvert, \nu x) \coloneqq k(\mu) - k(\nu)$, and this cocycle induces a generalised gauge action $\gamma^{E,k}\colon \R \curvearrowright C^*(\G_E)$, as defined in \cref{eq:action-induced-by-cocycle}. If $f\colon \partial E \to \R$ is the continuous map given by $f\restr{Z(e)} \equiv k(e)$ for all $e \in E^1$, then the induced cocycle $c_f$ defined in \cref{eq:cocycle-induced-by-function} agrees with $c_k$ on $\G_E$, and thus our weighted actions include all generalised gauge actions.
\end{example}

\begin{example}
Let $E$ be a topological graph, let $\partial E$ be the boundary-path space of $E$ defined in \cite[Definition~3.1]{KL2017}, and let $\sigma\colon \partial E\setminus E^0_{\operatorname{sg}}\to \partial E$ be the one-sided shift map mentioned in \cite[Lemma~6.1]{KL2017}. It follows from \cite[Proposition~3.6 and Lemma~6.1]{KL2017} and \cite[Theorem~3.16 and Proposition~4.4]{Ye07} that $(\partial E,\sigma)$ is a Deaconu--Renault system. The Deaconu--Renault groupoid $\G_{\partial E}$ of $(\partial E,\sigma)$ is the groupoid $\Gamma(\partial E,\sigma)$ defined in \cite[Definition~6.6]{KL2017}. It follows from \cite[Theorem~6.7]{KL2017} that $C^*(\G_{(\partial E,\sigma)})$ is isomorphic to the C*-algebra $\OO_E$ introduced in \cite{Katsura2004}. One can check that the isomorphism given in \cite[Theorem~6.7]{KL2017} intertwines the canonical gauge actions of $C^*(\G_{\partial E})$ and $\OO_E$.
\end{example}

\begin{example}
Let $(\X,\sigma_\X)$ be a one-sided shift of finite type (see, for instance, \cite[Section~13.8]{dlbm:isdc}). It follows from \cite[Theorem~1]{siyt:msrb} that $(\X,\sigma_\X)$ is a Deaconu--Renault system. The Deaconu--Renault groupoid $\G_\X$ of $(\X,\sigma_\X)$ is identical to the groupoid $\G_\X$ described in \cite[Section~2.6]{CEOR}.

If $A$ is an $n \times n$ $\{0,1\}$-matrix and $(\X_A,\sigma_{\X_A})$ is the one-sided topological Markov shift defined by $A$ (see, for instance, \cite[Page~3]{Kitchens}), then the Deaconu--Renault groupoid $\G_{\X_A}$ of $(\X_A,\sigma_{\X_A})$ is the groupoid $G_A$ described in \cite[Section~2.2]{MM14}, and $C^*(\G_{\X_A})$ is isomorphic to the Cuntz--Krieger algebra $\OO_A$ \cite{CK80} with generators $s_1, \dotsc, s_n$, via an isomorphism that maps $C(\X_A)$ onto the C*-subalgebra $\D_A$ generated by the projections $s_is_i^*$, and intertwines the canonical gauge action $\gamma^{\X_A}\colon \T\curvearrowright C^*(\G_{\X_A})$ and the gauge action $\lambda\colon \T \curvearrowright \OO_A$ (see, for instance, \cite[Section~2.3]{MM14}). Moreover, the isomorphism between $C^*(\G_{\X_A})$ and $\OO_A$ can be constructed such that it has the property that if $g\in C(\X_A,\Z)$, then it intertwines the weighted gauge action $\gamma^{\X_A,g} \colon \T \curvearrowright C^*(\G_{\X_A})$ and the action $\rho^{A,g}\colon \T \curvearrowright \OO_A$ considered in \cite{Matsumoto2021a}.

The map $(\sigma_A)^* \colon C(\X_A) \to C(\X_A)$ coincides with the map $\phi_A\colon \D_A \to \D_A$ given by $\phi_A(x) \coloneqq \sum_{i=1}^n s_i x s_i^*$ for $x \in \D_A$. This map $\phi_A$ appeared in Cuntz and Krieger's original paper as an invariant of one-sided conjugacy; cf.~\cite[Proposition~2.17]{CK80}. On the other hand, the map $(\sigma_A)_*\colon C(\X_A) \to C(\X_A)$ coincides with the adjacency operator $\lambda_A$ on $\D_A$, given by $\lambda_A(x) \coloneqq \sum_{i=1}^n s_i^* x s_i$ for $x \in \D_A$. \Cref{prop:sigma} shows that these maps can be used to characterise conjugacy.
\end{example}

\begin{example} \label{ex:oss}
Let $(\X,\sigma_\X)$ be a one-sided shift space, and let $(\widetilde{\X},\sigma_{\widetilde{\X}})$ be the cover of $(\X,\sigma_\X)$ constructed in \cite[Section~2.1]{BC20b}. It follows from \cite[Lemma~2.3]{BC20b} that $(\widetilde{\X},\sigma_{\widetilde{\X}})$ is a Deaconu--Renault system. If two one-sided shift spaces $(\X,\sigma_\X)$ and $(\Y,\sigma_\Y)$ are conjugate, then the corresponding Deaconu--Renault systems $(\widetilde{\X},\sigma_{\widetilde{\X}})$ and $(\widetilde{\Y},\sigma_{\widetilde{\Y}})$ are conjugate (see \cite[Lemma~4.1]{BC20b})l however, there are examples of non-conjugate one-sided shift spaces $(\X,\sigma_\X)$ and $(\Y,\sigma_\Y)$ for which $(\widetilde{\X},\sigma_{\widetilde{\X}})$ and $(\widetilde{\Y},\sigma_{\widetilde{\Y}})$ are conjugate (for example, consider a one-sided strictly sofic shift $(\X,\sigma_\X)$ and the one-sided edge shift $(\Y,\sigma_\Y)$ of its left Krieger cover, cf.~\cite[Exercise~6.1.9]{Kitchens}).

The Deaconu--Renault groupoid $\G_{\widetilde{\X}}$ of $(\widetilde{\X},\sigma_{\widetilde{\X}})$ is the groupoid $\G_\X$ described in \cite[Section~2.2]{BC20b}. It is shown in \cite{C04} that there is an isomorphism from $C^*(\G_{\widetilde{\X}})$ to the C*-algebra $\OO_\X$ studied in \cite{C08} that maps $C_0(\widetilde{\X})$ onto the C*-subalgebra $\D_\X$ and intertwines the canonical gauge actions of $C^*(\G_{\widetilde{\X}})$ and $\OO_\X$.
\end{example}

\begin{example}
Let $\X$ be a compact metrisable space, let $\phi\colon \X \to \X$ be a continuous locally injective surjection, and let $(\widehat{D},\psi)$ be the canonical extension of $(\X,\phi)$ constructed in \cite[Section~4]{Thomsen2011}. It follows from \cite[Proposition~4.1]{Thomsen2011} that $(\widehat{D},\psi)$ is a Deaconu--Renault system. If two continuous locally injective surjective maps $\phi\colon \X \to \X$ and $\phi'\colon \X' \to \X'$ between compact metrisable spaces are conjugate, then the corresponding Deaconu--Renault systems $(\widehat{D},\psi)$ and $(\widehat{D}',\psi')$ are conjugate (see \cite[Section~4]{Thomsen2011}), but there are, as in \cref{ex:oss}, examples of non-conjugate maps $\phi\colon \X \to \X$ and $\phi'\colon \X' \to \X'$ for which $(\widehat{D},\psi)$ and $(\widehat{D}',\psi')$ are conjugate.

The Deaconu--Renault groupoid $\G_{\widehat{D}}$ of $(\widehat{D},\psi)$ is the groupoid $\Gamma_\psi$ studied in \cite{Thomsen2011}. It therefore follows from \cite[Theorem~5.4]{Thomsen2011} that there is an isomorphism from $C^*(\G_{\widehat{D}})$ to the C*-algebra $C_r^*(\Gamma_\phi)$ constructed in \cite{Thomsen2010} that maps $C(\widehat{D})$ onto the C*-subalgebra $D_{\Gamma_\phi}$.
\end{example}

\section{Characterising conjugacy via groupoids and C*-algebras} \label{sec:characterisation}
In this section, we investigate the conditions that must be imposed on isomorphisms of Deaconu--Renault groupoids or their C*-algebras in order to ensure that the underlying Deaconu--Renault systems are conjugate. Our results are summarised in the following theorem.

\begin{theorem} \label{cor:summary}
Let $(\X,\sigma_\X)$ and $(\Y,\sigma_\Y)$ be second-countable Deaconu--Renault systems. The following statements are equivalent.
\begin{enumerate}[label=(\arabic*)]
\item \label{3.1(1)} The systems $(\X,\sigma_\X)$ and $(\Y,\sigma_\Y)$ are conjugate.

\item \label{3.1(2)} There exists a $*$-isomorphism $\varphi\colon C_0(\X) \to C_0(\Y)$ satisfying the following three equivalent conditions:
\begin{enumerate}[label=(\roman*)]
\item \label{3.1(2)(i)} there is a conjugacy $h\colon \X \to \Y$ satisfying $\varphi(f) = f\circ h^{-1}$ for $f\in C_0(\X)$;
\item \label{3.1(2)(ii)} $\varphi\big( (\sigma_\X)^*(f) g \big) = (\sigma_\Y)^*\big( \varphi(f) \big) \varphi(g)$ for all $f, g \in C_0(\X)$;
\item \label{3.1(2)(iii)} $\varphi\big(C_c(\domsigma{\X})\big) = C_c(\domsigma{\Y})$, and $\varphi \circ (\sigma_\X)_* = (\sigma_\Y)_* \circ \varphi\restr{C_c(\domsigma{\X})}$.
\end{enumerate}

\item \label{3.1(3)} There exists a groupoid isomorphism $\psi\colon \G_\X \to \G_\Y$ satisfying the following three equivalent conditions:
\begin{enumerate}[label=(\roman*)]
\item \label{3.1(3)(i)} there is a conjugacy $h\colon\X\to\Y$ such that $\psi(x,p,y)=(h(x),p,h(y))$ for $(x,p,y)\in\G_\X$;
\item \label{3.1(3)(ii)} $c_{g\circ \psi^{(0)}} = c_g \circ \psi$ for all $g \in C(\Y,\R)$;
\item \label{3.1(3)(iii)} there is a homeomorphism $h\colon\X\to\Y$ that satisfies $c_{g\circ h} = c_g \circ \psi$ for all $g \in C(\Y,\R)$.
\end{enumerate}

\item \label{3.1(4)} There is a $*$-isomorphism $\varphi\colon C^*(\G_\X) \to C^*(\G_\Y)$ satisfying the following two equivalent conditions:
\begin{enumerate}[label=(\roman*)]
\item \label{3.1(4)(i)} $\varphi(C_0(\X))= C_0(\Y)$, and there is a conjugacy $h\colon \X \to \Y$ such that $\varphi(f)=f\circ h^{-1}$ for $f\in C_0(\X)$ and $\varphi \circ \gamma_t^{\X, g\circ h} = \gamma_t^{\Y, g} \circ \varphi$ for all $t \in \R$ and $g \in C(\Y,\R)$;
\item \label{3.1(4)(ii)} there is a homeomorphism $h\colon \X \to \Y$ (which is not necessarily a conjugacy) such that $\varphi \circ \gamma_t^{\X, g\circ h} = \gamma_t^{\Y, g} \circ \varphi$ for all $t \in \R$ and $g \in C(\Y,\R)$.
\end{enumerate}
\end{enumerate}
\end{theorem}

\begin{remark}
\cref{cor:summary} follows from \cref{prop:sigma,prop:groupoid-conjugacy,prop:gauge-intertwining}, which we prove below (\cref{prop:sigma} gives us that the three conditions in \cref{cor:summary}\cref{3.1(2)} are equivalent, and that \cref{3.1(1),3.1(2)} are equivalent; \cref{prop:groupoid-conjugacy} gives us that the three conditions in \cref{cor:summary}\cref{3.1(3)} are equivalent, and that \cref{3.1(1),3.1(3)} are equivalent; and \cref{prop:gauge-intertwining} gives us that the two conditions in \cref{cor:summary}\cref{3.1(4)} are equivalent, and that \cref{3.1(1),3.1(4)} are equivalent).

It follows from \cref{prop:groupoid-conjugacy,prop:gauge-intertwining} that in \cref{3.1(3),3.1(4)} in \cref{cor:summary} we can replace $\R$ by any group that is \definitionemph{separating} for $\X$ and $\Y$, in the sense of \cref{def:separating}. In particular, if $\X$ and $\Y$ are totally disconnected, then we can replace $\R$ by $\Z$. We therefore obtain the following corollary, which is a generalisation and a strengthening of \cite[Theorem~1]{Matsumoto2021a}.
\end{remark}

\begin{corollary} \label{cor:graphs}
Let $E$ and $F$ be countable directed graphs. If there exist a homeomorphism $h\colon \partial E \to \partial F$ and a $*$-isomorphism $\varphi\colon C^*(E) \to C^*(F)$ satisfying
\[
\varphi\circ \gamma_z^{E, g\circ h} = \gamma_z^{F, g}\circ \varphi
\]
for all $z\in \T$ and $g\in C(\partial F, \Z)$, then the boundary-path spaces $(\partial E, \sigma_E)$ and $(\partial F, \sigma_F)$ are conjugate. Conversely, if $h\colon \partial E \to \partial F$ is a conjugacy, then there is a $*$-isomorphism $\varphi\colon C^*(E) \to C^*(F)$ satisfying $\varphi(\D(E))=\D(F)$, $\varphi(f)=f\circ h^{-1}$ for $f\in \D(E)$, and
\[
\varphi\circ \gamma_z^{E, g\circ h} = \gamma_z^{F, g}\circ \varphi
\]
for all $z\in \T$ and $g\in C(\partial F, \Z)$.
\end{corollary}

We now prove \cref{prop:sigma}, from which it follows that the three conditions \cref{3.1(2)(i),3.1(2)(ii),3.1(2)(iii)} in \cref{cor:summary}\cref{3.1(2)} are equivalent, and that \cref{3.1(1),3.1(2)} in \cref{cor:summary} are equivalent.

\begin{proposition} \label{prop:sigma}
Let $(\X,\sigma_\X)$ and $(\Y,\sigma_\Y)$ be Deaconu--Renault systems, and let $h\colon \X \to \Y$ be a homeomorphism. The map $\varphi\colon f \mapsto f \circ h^{-1}$ is a $*$-isomorphism from $C_0(\X)$ to $C_0(\Y)$, and the following statements are equivalent.
\begin{enumerate}[label=(\arabic*)]
\item \label{item:h-conjugacy} $h\colon \X \to \Y$ is a conjugacy.
\item \label{item:upper-star} For all $f, g \in C_0(\X)$,
\begin{equation} \label{eq:upper-star}
\varphi\big( (\sigma_\X)^*(f) g \big) = (\sigma_\Y)^*\big( \varphi(f) \big) \varphi(g).
\end{equation}
\item \label{item:lower-star} $\varphi\big(C_c(\domsigma{\X})\big) = C_c(\domsigma{\Y})$, and
\begin{equation} \label{eq:lower-star}
\varphi \circ (\sigma_\X)_* = (\sigma_\Y)_* \circ \varphi\restr{C_c(\domsigma{\X})}.
\end{equation}
\end{enumerate}
\end{proposition}

\begin{proof}
A routine argument shows that the map $\varphi\colon f \mapsto f \circ h^{-1}$ is a $*$-isomorphism from $C_0(\X)$ to $C_0(\Y)$.

\cref{item:h-conjugacy}$\implies$\big(\cref{item:upper-star}~and~\cref{item:lower-star}\big): Assume first that $h\colon \X \to \Y$ is a conjugacy. It follows from \cref{lem:alternative-conjugacy} that $h(\domsigma{\X}) = \domsigma{\Y}$, and hence $\varphi\big(C_c(\domsigma{\X})\big) = C_c(\domsigma{\Y})$.

Fix $f, g \in C_0(\X)$. Since $h$ is a conjugacy, we have
\[
f\big(\sigma_\X(h^{-1}(y))\big) \, g(h^{-1}(y)) = f\big(h^{-1}(\sigma_\Y(y))\big) \, g(h^{-1}(y))
\]
for all $y \in \domsigma{\Y}$, and \cref{eq:upper-star} follows. By \cref{lem:alternative-conjugacy}, we have $\sigma_\X^{-1}(h^{-1}(y)) = h^{-1}(\sigma_\Y^{-1}(y))$ for all $y \in \Y$, and hence
\[
\varphi\big((\sigma_\X)_*(f)\big)(y) \,=\, \sum_{z \in \sigma_\X^{-1}(h^{-1}(y))} f(z) \,=\, \sum_{w \in \sigma_\Y^{-1}(y)} f(h^{-1}(w)) \,=\, (\sigma_\Y)_*\big(\varphi(f)\big)(y),
\]
and so \cref{eq:lower-star} holds.

\cref{item:upper-star}$\implies$\cref{item:h-conjugacy}: Suppose that \cref{eq:upper-star} holds. Fix $y \in \domsigma{\Y}$. We claim that $\sigma_\X(h^{-1}(y)) = h^{-1}(\sigma_\Y(y))$. To see this, first choose $g \in C_0(\X)$ such that $g(h^{-1}(y)) = 1$. Then for all $f \in C_0(\X)$, we have
\[
f\big(\sigma_\X(h^{-1}(y))\big) = \varphi\big( (\sigma_\X)^*(f) g \big) = (\sigma_\Y)^*\big( \varphi(f) \big) \varphi(g) = f\big(h^{-1}(\sigma_\Y(y))\big),
\]
and so $\sigma_\X(h^{-1}(y)) = h^{-1}(\sigma_\Y(y))$, as claimed. Since the assumption of \cref{eq:upper-star} is equivalent to the assumption that
\[
\varphi^{-1}\big((\sigma_\Y)^*(f')g'\big) = (\sigma_\X)^*\big(\varphi^{-1}(f')\big) \varphi^{-1}(g')
\]
for all $f', g' \in C_0(\Y)$, a similar argument shows that $\sigma_\Y(h(x)) = h(\sigma_\X(x))$ for all $x \in \domsigma{\X}$. Therefore, $h$ is a conjugacy.

\cref{item:lower-star}$\implies$\cref{item:h-conjugacy}: Suppose that $\varphi\big(C_c(\domsigma{\X})\big) = C_c(\domsigma{\Y})$ and that \cref{eq:lower-star} holds. We will use the implication \cref{item:inverse-images}$\implies$\cref{item:conjugacy} of \cref{lem:alternative-conjugacy} to show that $h$ is a conjugacy. Fix $y \in \Y$. By \cref{eq:lower-star}, we have
\begin{equation} \label{eq:sigma_*}
\sum_{z \in \sigma_\X^{-1}(h^{-1}(y))} f(z) \,=\, \sum_{w \in \sigma_\Y^{-1}(y)} f(h^{-1}(w)) \,=\, \sum_{t \in h^{-1}(\sigma_\Y^{-1}(y))} f(t),
\end{equation}
for all $f \in C_c(\domsigma{\X})$. Suppose for contradiction that $h^{-1}(\sigma_\Y^{-1}(y)) \ne \sigma_\X^{-1}(h^{-1}(y))$. Then there exists $x \in \sigma_\X^{-1}(h^{-1}(y))$ such that $x \notin h^{-1}(\sigma_\Y^{-1}(y))$. Since $\X$ is a locally compact Hausdorff space, it is regular, and thus since $h^{-1}(\sigma_\Y^{-1}(y))$ is closed, there is an open neighbourhood $U \subset \X$ of $x$ such that $U \cap h^{-1}(\sigma_\Y^{-1}(y)) = \varnothing$. By Urysohn's lemma, there exists a function $f \in C_c(\X, [0,1])$ such that $f(x) = 1$ and $f(z) = 0$ for all $z \in h^{-1}(\sigma_\Y^{-1}(y))$. But this contradicts \cref{eq:sigma_*}, and so we must have $h^{-1}(\sigma_\Y^{-1}(y)) = \sigma_\X^{-1}(h^{-1}(y))$. Since the assumption of \cref{eq:lower-star} is equivalent to the assumption that
\[
\varphi^{-1} \circ (\sigma_\Y)_* = (\sigma_\X)_* \circ \varphi^{-1}\restr{C_c(\domsigma{\Y})},
\]
a similar argument shows that $h(\sigma_\X^{-1}(x)) = \sigma_\Y^{-1}(h(x))$ for all $x \in \X$. Therefore, \cref{lem:alternative-conjugacy} implies that $h$ is a conjugacy.
\end{proof}

We now introduce the notion of a \definitionemph{separating} group, which we use in \cref{prop:groupoid-conjugacy,prop:gauge-intertwining}.

\begin{definition} \label{def:separating}
Let $\X$ be a locally compact Hausdorff space, and let $\Gamma$ be a locally compact group with identity element $\id_\Gamma$. We say that $\Gamma$ is \definitionemph{separating} for $\X$ if, for any finite set $F \subset \X$ and $x \in \X {\setminus} F$, there exists $f \in C(\X, \Gamma)$ such that $f(x)$ has infinite order in $\Gamma$ and $f\restr{F} \equiv \id_\Gamma$.
\end{definition}

\begin{example}
Urysohn's lemma for locally compact Hausdorff spaces ensures that $\R$ is separating for any locally compact Hausdorff space $\X$. If $\X$ is totally disconnected, then $\Z$ is separating for $\X$.
\end{example}

\begin{lemma}\label{lem:separating}
Let $(\X,\sigma_\X)$ be a Deaconu--Renault system and suppose that $\Gamma$ is a locally compact group that is separating for $\X$. If $k,l\in \N$ and
\begin{equation} \label{eq:separating-sums}
\sum_{i=0}^k f(a_i) = \sum_{j=0}^l f(b_j)
\end{equation}
for some (not necessarily distinct) elements $a_0,\dotsc,a_k,b_0,\dotsc,b_l\in \X$ and all $f \in C(\X, \Gamma)$, then $k = l$ and $\{ a_0, \dotsc, a_k \} = \{ b_0, \dotsc, b_k \}$. Moreover, if $x, x' \in \X$ satisfy $\sigma_\X^k(x) = \sigma_\X^l(x')$ and $f^{(k)}(x) = f^{(l)}(x')$ for all $f\in C(\X, \Gamma)$, then $k = l$ and $x=x'$.
\end{lemma}

\begin{proof}
Let $A \coloneqq \{ a_0, \dotsc, a_k \}$ and $B \coloneqq \{ b_0, \dotsc, b_l \}$. For $x \in \X$, we may choose $f\in C(\X, \Gamma)$ such that $f(x)$ has infinite order and $f\restr{(A\cup B)\setminus \{x\}} \equiv \id_\Gamma$. By~\cref{eq:separating-sums}, we then have that
\[
\left\lvert \big\{i \in \{0,\dotsc,k\} : a_i = x \big\} \right\rvert = \left\lvert \big\{ j \in \{0,\dotsc,l\} : b_j = x \big\} \right\rvert.
\]
By applying this observation for all $x \in \X$, we see that $k = l$ and $A = B$.

For the second part, the hypothesis that $f^{(k)}(x) = f^{(l)}(x')$ for all $f \in C(\X, \Gamma)$ means that
\[
\sum_{i=0}^{k-1} f(\sigma_\X^i(x)) = \sum_{j=0}^{l-1} f(\sigma_\X^j(x'))
\]
for all $f\in C(\X, \Gamma)$. It follows from the first part that $k = l$ and that
\[
\{ \sigma_\X^i(x) : i = 0,\dotsc,k-1\} = \{ \sigma_\X^j(x') : j= 0,\dotsc,k-1\}.
\]
This means that $x = \sigma_\X^p(x')$ and $x' = \sigma_\X^q(x)$ for some $p,q\in \N$. By choosing $n \in \N_+$ such that $k \le n(p+q)$, the fact that $\sigma_\X^k(x) = \sigma_\X^k(x')$ implies that
\[
x = \sigma_\X^{n(p+q)}(x) = \sigma_\X^{n(p+q)}(x') = x',
\]
as required.
\end{proof}

We now prove \cref{prop:groupoid-conjugacy}, from which it follows that the three conditions \cref{3.1(3)(i),3.1(3)(ii),3.1(3)(iii)} in \cref{cor:summary}\cref{3.1(3)} are equivalent, and that \cref{3.1(1),3.1(3)} in \cref{cor:summary} are equivalent.

\begin{proposition} \label{prop:groupoid-conjugacy}
Let $(\X,\sigma_\X)$ and $(\Y,\sigma_\Y)$ be Deaconu--Renault systems.
A conjugacy $h\colon \X \to \Y$ induces a groupoid isomorphism $\psi\colon \G_\X \to \G_\Y$ satisfying
\[
\psi(x, p, y) = (h(x), p, h(y)),
\]
for $(x,p,y)\in \G_\X$.
Moreover, if $\psi\colon \G_\X \to \G_\Y$ is a groupoid isomorphism and $\Gamma$ is a locally compact group that is separating for $\Y$, then the following three conditions are equivalent.
\begin{enumerate}[label=(\arabic*)]
\item \label{3.8(1)} There is a conjugacy $h\colon \X \to \Y$ such that $\psi(x, p, y) = (h(x), p, h(y))$ for $(x,p,y)\in \G_\X$.
\item \label{3.8(2)} $c_{g\circ \psi^{(0)}} = c_g \circ \psi$ for $g \in C(\Y, \Gamma)$.
\item \label{3.8(3)} There is a homeomorphism $h\colon \X \to \Y$ that satisfies $c_{g \circ h} = c_g \circ \psi$ for all $g \in C(\Y, \Gamma)$.
\end{enumerate}
\end{proposition}

\begin{proof}
A routine argument shows that if $h\colon \X \to \Y$ is a conjugacy, then the map $\psi\colon \G_\X \to \G_\Y$ given by $\psi(x, p, y) = (h(x), p, h(y))$ is a groupoid isomorphism.

We now prove the implication $\cref{3.8(1)}\implies\cref{3.8(2)}$. Suppose that condition~$\cref{3.8(1)}$ holds. Let $\Gamma$ be a locally compact group, and fix $g \in C(\Y, \Gamma)$. We claim that $c_{g \circ h} = c_g \circ \psi$. It suffices to prove the relation for groupoid elements in $Z(\X, 1, 0, \sigma_\X(\X)) \subset \G_\X$, so fix $x \in \domsigma{\X}$. Then
\[
c_{g\circ h}\big(x, 1, \sigma_\X(x)\big) = g(h(x)) = c_g\big(h(x), 1, \sigma_\Y(h(x))\big) = c_g\big(\psi(x, 1, \sigma_\X(x))\big),
\]
which proves the claim. The implication $\cref{3.8(2)}\implies\cref{3.8(3)}$ is obvious.

It remains to prove $\cref{3.8(3)}\implies\cref{3.8(1)}$. Suppose that $\psi\colon \G_\X \to \G_\Y$ is a groupoid isomorphism and that $h\colon \X \to \Y$ is a homeomorphism such that $c_{g \circ h} = c_g \circ \psi$ for all $g \in C(\Y, \Gamma)$. Consider the homeomorphism $\tilde{h} \coloneqq \psi^{(0)}$ and note that $h$ and $\tilde{h}$ need not be equal. We will show that $\tilde{h}$ is a conjugacy.

Since $\Gamma$ is separating for $\Y$, it contains an element of infinite order, and this element generates a copy of $\Z$ in $\Gamma$. By choosing $g\in C(\Y, \Gamma)$ to be constantly equal to such an element, we see that $\psi$ intertwines the canonical cocycles. Therefore,
\[
\psi(x, p, y) = (\tilde{h}(x), p, \tilde{h}(y))
\]
for $(x,p,y)\in \G_\X$, and it follows from~\cite[Theorem~8.10]{CRST} that $\tilde{h}$ and $\tilde{h}^{-1}$ are eventual conjugacies (cf.~\cite[Definition~8.9]{CRST}). In particular, $\tilde{h}(\domsigma{\X})=\domsigma{\Y}$. For $x\in \domsigma{\X}$, we let $\tilde{k}(x)$ be the minimal nonnegative integer satisfying
\begin{equation}\label{eq:k}
\big(\sigma_\Y^{\tilde{k}(x)+1}\circ\tilde{h}\big)(x) = \big(\sigma_\Y^{\tilde{k}(x)}\circ\tilde{h}\circ\sigma_\X\big)(x),
\end{equation}
noting that such an integer exists because $\tilde{h}$ is an eventual conjugacy. By~\cref{lem:alternative-conjugacy} it only remains to show that $\tilde{h}(\sigma_\X(x))=\sigma_\Y(\tilde{h}(x))$ for $x\in \domsigma{\X}$. This is equivalent to showing that $\tilde{k}(x) = 0$ for $x\in \domsigma{\X}$.

Fix $x\in \domsigma{\X}$, and suppose for contradiction that $\tilde{k}(x)>0$. The hypothesis $c_{g \circ h} = c_g \circ \psi$ implies that
\begin{equation} \label{eq:dagger}
g(h(x)) = c_g\big(\psi(x,1,\sigma_\X(x))\big) = \sum_{i=0}^{\tilde{k}(x)} g\big( (\sigma_\Y^i \circ \tilde{h})(x) \big) - \sum_{j=0}^{\tilde{k}(x)-1} g\big( (\sigma_\Y^j \circ \tilde{h} \circ \sigma_\X)(x) \big),
\end{equation}
for any $g\in C(\Y, \Gamma)$. Since $\Gamma$ is separating for $\Y$, it follows from~\cref{lem:separating} that the sets
\[
A \coloneqq \big\{ (\sigma_\Y^i \circ \tilde{h})(x) : i = 0,\dotsc, \tilde{k}(x) \big\} \, \text{ and } \, B \coloneqq \big\{ h(x), \, (\sigma_\Y^j \circ \tilde{h} \circ \sigma_\X)(x) : j = 0,\dotsc, \tilde{k}(x)-1 \big\}
\]
are equal. Therefore, there exists $i \in \{0,\dotsc, \tilde{k}(x)\}$ such that
\begin{equation} \label{eq:i}
\big(\sigma_\Y^i \circ \tilde{h}\big)(x) = \big(\sigma_\Y^{\tilde{k}(x)-1} \circ \tilde{h} \circ \sigma_\X\big)(x).
\end{equation}
If $i = \tilde{k}(x)$, then~\cref{eq:i} contradicts the minimality of $\tilde{k}(x)$, so we must have $i < \tilde{k}(x)$.

We will now show that $\tilde{k}(x)=0$. Let us first consider the case when $\tilde{h}(x)$ is aperiodic, i.e.~there is no pair of distinct nonnegative integers $m$ and $n$ such that $(\sigma_\Y^m \circ \tilde{h})(x) = (\sigma_\Y^n \circ \tilde{h})(x)$. If $\tilde{k}(x) > 0$, then
\[
\big(\sigma_\Y^{\tilde{k}(x)+1}\circ\tilde{h}\big)(x) = \big(\sigma_\Y^{\tilde{k}(x)}\circ\tilde{h}\circ\sigma_\X\big)(x) = \big(\sigma_\Y^{i+1}\circ\tilde{h}\big)(x),
\]
and this together with the assumption that $\tilde{h}(x)$ is aperiodic implies that $i = \tilde{k}(x)$, which we have already seen cannot be the case. Therefore, $\tilde{k}(x) = 0$ when $\tilde{h}(x)$ is aperiodic.

We now consider the case when $\tilde{h}(x)$ is eventually periodic. In this case, there is a nonnegative integer $n$ such that $(\sigma_\Y^n\circ \tilde{h})(x)$ is periodic, and we proceed by induction on $n$. For $n = 0$ (i.e.~$\tilde{h}(x)$ is periodic), we choose $i$ as in \cref{eq:i}, and observe that $\big(\sigma_\Y^i \circ \tilde{h}\big)(x) = \big(\sigma_\Y^{\tilde{k}(x)-1} \circ \tilde{h} \circ \sigma_\X\big)(x)$ is periodic.

If $p$ is a period of $\big(\sigma_\Y^i \circ \tilde{h}\big)(x) = \big(\sigma_\Y^{\tilde{k}(x)-1} \circ \tilde{h} \circ \sigma_\X\big)(x)$, then the eventual conjugacy condition~\cref{eq:k} implies that
\[
\big(\sigma_\Y^{\tilde{k}(x)} \circ \tilde{h}\big)(x) = \big(\sigma_\Y^{\tilde{k}(x)+p} \circ \tilde{h}\big)(x) = \big(\sigma_\Y^{\tilde{k}(x)-1+p} \circ \tilde{h} \circ \sigma_\X\big)(x) = \big(\sigma_\Y^{\tilde{k}(x)-1} \circ \tilde{h} \circ \sigma_\X\big)(x),
\]
which again contradicts the minimality of $\tilde{k}(x)$. Therefore, $\tilde{k}(x) = 0$ when $\tilde{h}(x)$ is periodic.

Assume now that $n\in\N$ and that $\tilde{k}(x')=0$ whenever $x'\in \domsigma{\X}$ and $\sigma_\Y^n(\tilde{h}(x'))$ is periodic. Suppose that $x \in \domsigma{\X}$ with $\sigma_\Y^{n+1}(\tilde{h}(x))$ periodic, and that $n$ is the minimal nonnegative integer for which $\sigma_\Y^{n+1}(\tilde{h}(x))$ is periodic. Assuming for contradiction that $\tilde{k}(x) > 0$, we again choose $i$ according to~\cref{eq:i}. Since $i < \tilde{k}(x)$, we have $\sigma_\Y^{i+1}(\tilde{h}(x)) \in A = B$. There are two cases to consider.

For the first case, suppose that
\[
\sigma_\Y^{i+1}(\tilde{h}(x)) = \sigma_\Y^j\big(\tilde{h}(\sigma_\X(x))\big),
\]
for some $j \in \{0,\dotsc,k(x)-1\}$. From~\cref{eq:i}, we see that
\[
\big(\sigma_\Y^j \circ \tilde{h} \circ \sigma_\X\big)(x) = \big(\sigma_\Y^{i+1} \circ \tilde{h}\big)(x) = \big(\sigma_\Y^{\tilde{k}(x)}\circ \tilde{h}\circ \sigma_\X\big)(x)
\]
is periodic. In particular, $\big(\sigma_\Y^{\tilde{k}(x) - 1} \circ \tilde{h} \circ \sigma_\X\big)(x)$ is periodic (because $j\le \tilde{k}(x) - 1$). Since $i < \tilde{k}(x)$, we also see that $\big(\sigma_\Y^{\tilde{k}(x)} \circ \tilde{h}\big)(x)$ is periodic with the same period as $\big(\sigma_\Y^{\tilde{k}(x) - 1}\circ \tilde{h} \circ \sigma_\X\big)(x)$. It now follows from~\cref{eq:k} that if $p$ is a common period of $\big(\sigma_\Y^{\tilde{k}(x)} \circ \tilde{h}\big)(x)$ and $\big(\sigma_\Y^{\tilde{k}(x) - 1} \circ \tilde{h} \circ \sigma_\X\big)(x)$, then
\[
\big(\sigma_\Y^{\tilde{k}(x)} \circ \tilde{h}\big)(x) = \big(\sigma_\Y^{\tilde{k}(x)+p} \circ \tilde{h}\big)(x) = \big(\sigma_\Y^{\tilde{k}(x)-1+p} \circ \tilde{h} \circ \sigma_\X\big)(x) = \big(\sigma_\Y^{\tilde{k}(x)-1} \circ \tilde{h} \circ \sigma_\X\big)(x),
\]
which contradicts the minimality of $\tilde{k}(x)$.

For the second case, suppose that $\sigma_\Y^{i+1}(\tilde{h}(x)) = h(x)$. Choose $x'\in \X$ such that $\tilde{h}(x') = h(x)$. If $n \ge i$, then $\sigma_\Y^{n-i}(\tilde{h}(x'))=\sigma_\Y^{n+1}(\tilde{h}(x))$ is periodic, and if $n < i$, then $\tilde{h}(x')=\sigma_\Y^{i+1}(\tilde{h}(x))$ is periodic. In both cases it follows from the inductive hypothesis that $\tilde{k}(x') = 0$. The assumption that $c_{g\circ h} = c_g \circ \psi$ then implies that $g(h(x')) = g(\tilde{h}(x'))$ for all $g\in C(\Y, \Gamma)$, and hence $h(x') = \tilde{h}(x')$. Since $\tilde{h}(x') = h(x)$ and $h$ is a homeomorphism, we have $x = x'$. This means that either $(\sigma_\Y^{n - i} \circ \tilde{h})(x)$ is periodic (if $n \ge i$), or that $\tilde{h}(x)$ is periodic (if $n < i$), but this contradicts the assumption that $n$ is the minimal nonnegative integer for which $\sigma_\Y^{n+1}(\tilde{h}(x))$ is periodic. We conclude that $\tilde{k}(x) = 0$ for all $x \in \domsigma{\X}$, and this implies that $\tilde{h}$ is a conjugacy.
\end{proof}

For the proof of \cref{prop:gauge-intertwining}, we need the following two lemmas. Given an automorphism $\gamma$ of $C^*(\G_\X)$, we define
\[
C^*(\G_\X)^\gamma \coloneqq \{ f \in C^*(\G_\X) : \gamma(f) = f \}.
\]

\begin{lemma} \label{lem:diagonal-fixed}
Let $(\X,\sigma_\X)$ be a Deaconu--Renault system, and let $\Gamma$ be a locally compact abelian group that is separating for $\X$. Then
\[
C_0(\X) = \bigcap_{f \in C(\X, \Gamma)} C^*(\G_\X)^{\gamma^{\X, f}}.
\]
\end{lemma}

\begin{proof}
Every function in $C_0(\X)$ is fixed by all the weighted automorphisms, so one containment is clear. For the reverse containment, take $\xi \in C_c(\G_\X)$ such that $\xi$ is fixed by $\gamma^{\X,f}$ for all $f \in C(\X, \Gamma)$. If $(x, k - l, y) \in \G_\X$ with $\sigma_\X^k(x) = \sigma_\X^l(y)$ and $\xi(x, k - l, y) \ne 0$, then
\[
\xi(x, k - l, y) = \gamma_\chi^{\X, f}(\xi)(x, k - l, y) = \chi\big(f^{(k)}(x) - f^{(l)}(y)\big) \, \xi(x, k - l, y),
\]
for all $\chi \in \widehat{\Gamma}$. Since the characters of an abelian group separate points, it follows that $f^{(k)}(x) = f^{(l)}(y)$ for every $f \in C(\X, \Gamma)$. Since $\Gamma$ is separating for $\X$, it follows that $k=l$ and $x = y$, by \cref{lem:separating}. Thus $\xi$ is only supported on the unit space of $\G_\X$, and so $\xi \in C_0(\X)$, and the result follows.
\end{proof}

The following technical lemma is actually the main bulk of the proof of~\cref{prop:gauge-intertwining} below. It uses the groupoid reconstruction theory of~\cite{CRST}. We state and prove \cref{lem:technical-lemma} in a more general setting than we need here, as we believe it may be of independent interest.

Recall that if $c$ is a cocycle from $\G$ into a group $G$ with identity element $\id_G$, then $c^{-1}(\id_G)$ is a subgroupoid of $\G$. We refer the reader to~\cite{CRST} for relevant details on the coactions $\delta_{c_1}$ and $\delta_{c_2}$. The reader is invited to let $c_1$ and $c_2$ be the canonical continuous cocycles on $\G_\X$ and $\G_\Y$, respectively; in which case, the coaction condition (\cref{eq:coaction}) reduces to the condition that $\varphi$ intertwines the canonical gauge actions.

Note that in the statement below, we are not assuming that $h$ and $\tilde{h}$ are equal.

\begin{lemma} \label{lem:technical-lemma}
Let $(\X,\sigma_\X)$ and $(\Y,\sigma_\Y)$ be second-countable Deaconu--Renault systems, and let $\tilde{h}\colon \X \to \Y$ be a homeomorphism. Let $G$ be a discrete group with identity element $\id_G$. Let $c_1 \colon \G_\X \to G$ and $c_2\colon \G_\Y \to G$ be continuous cocycles such that $\Iso(c_1^{-1}(\id_G))^\circ = \X$ and $\Iso(c_2^{-1}(\id_G))^\circ = \Y$. Suppose that $\varphi\colon C^*(\G_\X) \to C^*(\G_\Y)$ is a $*$-isomorphism such that $\varphi(C_0(\X)) = C_0(\Y)$ and $\varphi(f) = f \circ \tilde{h}^{-1}$ for all $f \in C_0(\X)$, and that $\varphi$ satisfies the coaction condition
\begin{equation} \label{eq:coaction}
\delta_{c_2} \circ \varphi = (\varphi\otimes \id) \circ \delta_{c_1}.
\end{equation}
Then there is a groupoid isomorphism $\psi\colon \G_\X \to \G_\Y$ satisfying $\psi^{(0)} = \tilde{h}$ and $c_1 = c_2 \circ \psi$. Moreover, this $\psi$ has the property that $c_{g\circ h} = c_g \circ \psi$ whenever $\Gamma$ is a locally compact abelian group, $g \in C(\Y, \Gamma)$, $h\colon \X \to \Y$ is a homeomorphism, and
\begin{equation} \label{eq:technical-lemma}
\varphi \circ \gamma_\chi^{\X, g\circ h} = \gamma_\chi^{\Y, g} \circ \varphi
\end{equation}
for all $\chi \in \widehat{\Gamma}$.
\end{lemma}

\begin{proof}
The first part of the lemma follows from~\cite[Theorem~6.2]{CRST}, but since we need an explicit description of the groupoid isomorphism $\psi\colon \G_\X \to \G_\Y$ in order to prove the second half of the lemma, we begin by recalling the construction of $\psi$.

For this, let us first establish some notation. The \definitionemph{open support} of $\xi \in C_c(\G_\X)$ is the set
\[
\osupp(\xi) \coloneqq \{\gamma \in \G_\X : \xi(\gamma) \ne 0 \}.
\]
The extended Weyl groupoid $\h_\X \coloneqq \h(C^*(\G_\X), C_0(\X), \delta_{c_1})$ of the triple $(C^*(\G_\X), C_0(\X), \delta_{c_1})$ consists of equivalence classes $[n,x]$ of pairs $(n,x)$, where $n$ is a normaliser of $C_0(\X)$ in $C^*(\G_\X)$ and $x \in \osupp(n)$; cf.~\cite[Section~4]{CRST}. Let $\h_\Y \coloneqq \h(C^*(\G_\Y), C_0(\Y), \delta_{c_2})$ be the extended Weyl groupoid of $(C^*(\G_\Y), C_0(\Y), \delta_{c_2})$. Let $\theta_\X \colon \G_\X \to \h_\X$ and $\theta_\Y\colon \G_\Y \to \h_\Y$ be the groupoid isomorphisms of~\cite[Proposition~6.5]{CRST}, and let $\varphi^*\colon \h_\X \to \h_\Y$ be the groupoid isomorphism given by $\varphi^*([n, x]) = [\varphi(n), \tilde{h}(x)]$ for $[n,x] \in \h_\X$; cf.~\cite[proof of Theorem~6.2]{CRST}. The composition
\[
\psi\coloneqq \theta_\Y^{-1} \circ \varphi^* \circ \theta_\X \colon \G_\X \to \G_\Y
\]
is then a groupoid isomorphism that satisfies $\psi^{(0)} = \tilde{h}$ and $c_1 = c_2 \circ \psi$; cf.~\cite[proof of Theorem~6.2]{CRST}.

We now observe that if $[n,x] \in \h_\Y$ and $\theta_\Y(\eta) = [n,x]$, then $n(\eta) \ne 0$. Indeed, following the proof of the fact that $\theta_\Y$ is surjective in \cite[proof of Proposition~6.5]{CRST}, there exist $\eta' \in \Iso({c_2}^{-1}(\id_G))^\circ$ and $\gamma \in \G_\Y$ with $s(\eta') = r(\eta') = s(\gamma) = x$ and $n(\gamma) \ne 0$ such that $\theta_\Y(\gamma (\eta')^{-1}) = [n,x]$. Our assumption that $\Iso({c_2}^{-1}(\id_G))^\circ = \Y$ implies that $\eta' = x$, and since $\theta_\Y$ is injective, it follows that $\gamma = \eta$, and, in particular, $n(\eta) \ne 0$.

Now, let $h\colon \X \to \Y$ be a homeomorphism (which need not be the same as $\tilde{h}$), let $\Gamma$ be a locally compact abelian group, and fix $g \in C(\Y, \Gamma)$ such that $\varphi \circ \gamma_\chi^{\X, g \circ h} = \gamma_\chi^{\Y, g} \circ \varphi$ for all $\chi \in \widehat{\Gamma}$. We need to show that $c_{g\circ h} = c_g \circ \psi$. Since $\G_\X$ is generated by elements either belonging to $\G_\X^{(0)}$ or to the compact open set $Z(\X, 1, 0, \sigma_\X(\X)) = \{ \big(x, 1, \sigma_\X(x)\big) : x\in \domsigma{\X}\}$, it suffices to verify that
\begin{equation} \label{eq:cocycles-eta}
(g \circ h)(r(\eta)) = c_g(\psi(\eta)),
\end{equation}
for $\eta \in Z(\X, 1, 0, \sigma_\X(\X))$.

Choose $n \in C_c(\G_\X)$ with $n(\eta) = 1$ such that $\osupp(n)$ is a bisection contained in $Z(\X,1,0,\sigma_\X(\X)) \cap {c_1}^{-1}(c_1(\eta))$. By the construction of $\psi$, we have
\[
\theta_\Y\big(\psi(\eta)\big) = \varphi^*\big(\theta_\X(\eta)\big) = [\varphi(n), \tilde{h}(s(\eta))],
\]
and so the observation above implies that $\varphi(n)(\psi(\eta)) \ne 0$. Moreover, since $\osupp(n) \subset Z(\X,1,0,\sigma_\X(\X))$, we have $\gamma_\chi^{\X, g\circ h}(n) = (\chi \circ g\circ h) n$ with $\chi \circ g\circ h \in C_b(\X)$. Here, we view $C_b(\X)$ as a subalgebra of the multiplier algebra $M(C^*(\G_\X))$.

There is a $*$-isomorphism of multiplier algebras $\tilde{\varphi}\colon M(C^*(\G_\X)) \to M(C^*(\G_\Y))$ which extends $\varphi$, and since the diagonal subalgebras contain approximate units of the ambient C*-algebras, we have $\tilde{\varphi}(C_b(\X)) = C_b(\Y)$ with $\tilde{\varphi}(f) = f \circ (\beta \tilde{h})^{-1}$, where $\beta \tilde{h}\colon \beta \X \to \beta \Y$ is the unique extension of $\tilde{h}$ to the Stone--\v{C}ech compactifications; c.f.~for example, \cite[Proposition~3.12.10, and~3.12.12]{Pedersen}.
It now follows that
\begin{align*}
\varphi\big(\gamma_\chi^{\X, g \circ h}(n)\big)(\psi(\eta)) &= \big(\tilde{\varphi}(\chi \circ g \circ h) \varphi(n)\big) (\psi(\eta)) \\
&= \chi\big( (g \circ h \circ \tilde{h}^{-1} \circ \psi^{(0)})(r(\eta))\big) \varphi(n)(\psi(\eta)) \\
&= \chi\big( (g \circ h)(r(\eta))\big) \varphi(n)(\psi(\eta)).
\end{align*}
Applying this observation together with \cref{eq:technical-lemma}, we see that
\begin{align*}
\chi\big( c_g(\psi(\eta))\big) \varphi(n)(\psi(\eta)) &= \gamma_\chi^{\Y, g}\big(\varphi(n)\big)(\psi(\eta)) \\
&= \varphi\big(\gamma_\chi^{\X, g\circ h}(n)\big)(\psi(\eta)) \\
&= \chi\big((g\circ h)(r(\eta))\big) \varphi(n)(\psi(\eta)),
\end{align*}
for all $\chi \in \widehat{\Gamma}$. Since $\varphi(n)(\psi(\eta)) \ne 0$, it follows that
\[
\chi\big((g\circ h)(r(\eta))\big) = \chi\big(c_g(\psi(\eta))\big),
\]
for all $\chi \in \widehat{\Gamma}$. Since the characters of an abelian group separate points, \cref{eq:cocycles-eta} follows.
\end{proof}

Before we get to \cref{prop:gauge-intertwining}, we point out that if $(\X,\sigma_\X)$ and $(\Y,\sigma_\Y)$ are second-countable topologically free Deaconu--Renault systems, then the coaction condition \cref{eq:coaction} in \cref{lem:technical-lemma} is superfluous. Although we do not need this fact in this paper, we believe it is worth recording in a corollary.

\begin{corollary}
Suppose that $(\X,\sigma_\X)$ and $(\Y,\sigma_\Y)$ are second-countable and topologically free Deaconu--Renault systems. If $\varphi\colon C^*(\G_\X) \to C^*(\G_\Y)$ is a $*$-isomorphism satisfying $\varphi(C_0(\X)) = C_0(\Y)$, then there is a groupoid isomorphism $\psi\colon \G_\X \to \G_\Y$ such that $\varphi(f)=f\circ(\psi^{(0)})^{-1}$, for $f\in C_0(\X)$. If, moreover, $\Gamma$ is a locally compact abelian group, $g \in C(\Y, \Gamma)$, $h\colon \X \to \Y$ is a homeomorphism, and $\varphi \circ \gamma_\chi^{\X, g\circ h} = \gamma_\chi^{\Y, g} \circ \varphi$ for all $\chi \in \widehat{\Gamma}$, then $c_{g\circ h} = c_g \circ \psi$.
\end{corollary}

\begin{proof}
By \cref{lem:top-free-implies-effective}, we have $\Iso(\G_\X)^\circ = \X$ and $\Iso(\G_\Y)^\circ = \Y$, and so the result follows immediately from \cref{lem:technical-lemma} by letting $G$ be the trivial group and taking $c_1\colon \G_\X \to G$ and $c_2\colon \G_\Y \to G$ to be the trivial cocycles.
\end{proof}

Finally, we prove \cref{prop:gauge-intertwining}, from which it follows that the two conditions \cref{3.1(4)(i),3.1(4)(ii)} in \cref{cor:summary}\cref{3.1(4)} are equivalent, and that \cref{3.1(1),3.1(4)} in \cref{cor:summary} are equivalent.

\begin{proposition} \label{prop:gauge-intertwining}
Let $(\X,\sigma_\X)$ and $(\Y,\sigma_\Y)$ be second-countable Deaconu--Renault systems.

\begin{enumerate}[label=(\roman*)]
\item \label{item:i} If $h\colon \X \to \Y$ is a conjugacy, then there is a $*$-isomorphism $\varphi\colon C^*(\G_\X) \to C^*(\G_\Y)$ satisfying $\varphi(C_0(\X))= C_0(\Y)$, $\varphi(f)=f\circ h^{-1}$ for all $f \in C_0(\X)$, and $\varphi \circ \gamma_\chi^{\X, g\circ h} = \gamma_\chi^{\Y, g} \circ \varphi$ whenever $\Gamma$ is a locally compact abelian group, $g \in C(\Y,\Gamma)$, and $\chi\in\widehat{\Gamma}$.
\item \label{item:ii} Conversely, suppose that $\varphi\colon C^*(\G_\X) \to C^*(\G_\Y)$ is a $*$-isomorphism, $h\colon \X \to \Y$ is a homeomorphism (which is not necessarily a conjugacy), and $\Gamma$ is a locally compact abelian group that is separating for $\X$ and $\Y$ such that $\varphi \circ \gamma_\chi^{\X, g \circ h} = \gamma_\chi^{\Y, g} \circ \varphi$ for all $\chi \in \widehat{\Gamma}$ and $g \in C(\Y,\Gamma)$. Then $\varphi(C_0(\X)) = C_0(\Y)$, and there exists a conjugacy $\tilde{h}\colon \X \to \Y$ such that $\varphi(f)=f\circ \tilde{h}^{-1}$ for all $f \in C_0(\X)$ and $\varphi \circ \gamma_\chi^{\X, g\circ \tilde{h}} = \gamma_\chi^{\Y, g} \circ \varphi$ for all $\chi \in \widehat{\Gamma}$ and $g \in C(\Y,\Gamma)$.
\end{enumerate}
\end{proposition}

\begin{proof}
For part~\cref{item:i}, suppose that $h\colon \X \to \Y$ is a conjugacy. By \cref{prop:groupoid-conjugacy}, there is a groupoid isomorphism $\psi\colon \G_\X \to \G_\Y$ satisfying
\[
\psi(x, p, y) = (h(x), p, h(y)),
\]
for $(x,p,y)\in \G_\X$. This isomorphism induces a $*$-isomorphism $\varphi\colon C^*(\G_\X) \to C^*(\G_\Y)$ satisfying $\varphi(\xi) = \xi \circ \psi^{-1}$ for $\xi \in C_c(\G_\X)$ and $\varphi(C_0(\X)) = C_0(\Y)$ with $\varphi(f) = f \circ h^{-1}$ for $f \in C_0(\X)$. Suppose that $\Gamma$ is a locally compact abelian group and $g\in C(\Y, \Gamma)$. It follows from \cref{prop:groupoid-conjugacy} that $c_{g\circ h} = c_g\circ \psi$. (Note that the proof of the relevant part of \cref{prop:groupoid-conjugacy} does not require $\Gamma$ to be separating for $\Y$.) This implies that $\varphi\circ \gamma_\chi^{\X, g\circ h} = \gamma_\chi^{\Y, g}\circ \varphi$ for all $\chi\in \widehat{\Gamma}$.

For part~\cref{item:ii}, suppose that $\varphi\colon C^*(\G_\X) \to C^*(\G_\Y)$ is a $*$-isomorphism, $h\colon \X \to \Y$ is a homeomorphism, and $\Gamma$ is a locally compact abelian group that is separating for $\X$ and $\Y$ such that $\varphi \circ \gamma_\chi^{\X, g\circ h} = \gamma_\chi^{\Y, g} \circ \varphi$ for all $\chi \in \widehat{\Gamma}$ and $g \in C(\Y,\Gamma)$. Since $\Gamma$ is separating for both $\X$ and $\Y$, it follows from \cref{lem:diagonal-fixed} that $\varphi(C_0(\X)) = C_0(\Y)$. Let $\tilde{h}\colon \X \to \Y$ be the induced homeomorphism satisfying $\varphi(f) = f\circ \tilde{h}^{-1}$ for $f\in C_0(\X)$ from Gelfand duality.

Since $\Gamma$ is separating for $\X$ and $\Y$, it contains an element $\zeta$ of infinite order. Choose $g\in C(\Y, \Gamma)$ to be constantly equal to $\zeta$. Then $\gamma_\chi^{\Y, g}=\gamma_{\chi(\zeta)}^\Y$ and $\gamma_\chi^{\X, g\circ h}=\gamma_{\chi(\zeta)}^\X$ for $\chi \in \widehat{\Gamma}$. Since $\varphi \circ \gamma_\chi^{\X, g\circ h} = \gamma_\chi^{\Y, g} \circ \varphi$ for all $\chi \in \widehat{\Gamma}$, it follows that if we let $G=\Z$, $c_1=c_\X$, and $c_2=c_\Y$, then \cref{eq:coaction} in \cref{lem:technical-lemma} holds. An application of \cref{lem:technical-lemma} thus gives us a groupoid isomorphism $\psi\colon \G_\X \to \G_\Y$ with $\psi^{(0)} = \tilde{h}$ such that $c_{g\circ h} = c_g \circ \psi$ for all $g \in C(\Y, \Gamma)$. It now follows from \cref{prop:groupoid-conjugacy} that $\tilde{h}=\psi^{(0)}$ is a conjugacy, and that $c_{g\circ \tilde{h}} = c_g \circ \psi$ for all $g \in C(\Y, \Gamma)$.

Fix $g \in C(\Y, \Gamma)$. Then $c_{g\circ \tilde{h}} = c_g \circ \psi = c_{g\circ h}$. It follows that $\gamma^{\X, g \circ \tilde{h}} = \gamma^{\X, g \circ h}$, and thus
\[
\varphi \circ \gamma_\chi^{\X, g \circ \tilde{h}} = \varphi \circ \gamma_\chi^{\X, g \circ h} = \gamma_\chi^{\Y, g} \circ \varphi
\]
for all $\chi \in \widehat{\Gamma}$.
\end{proof}

\begin{remark}
In~\cite{ERS20}, Ruiz, Sims, and the fourth-named author show that a pair of amplified graphs (i.e.~graphs in which every vertex emits either infinitely many or no edges to any other vertex) are graph-isomorphic if and only if there is a $*$-isomorphism of their graph C*-algebras that intertwines the canonical gauge actions. It follows from this and \cref{prop:gauge-intertwining} that the boundary-path spaces of two amplified graphs are conjugate if and only if there is a $*$-isomorphism of their graph C*-algebras that intertwines the canonical gauge actions. This is an interesting result which we cannot expect to hold for larger classes of graphs. In fact, it is known that if the boundary-path spaces of two directed graphs are \definitionemph{eventually conjugate}, then there is a $*$-isomorphism of their graph C*-algebras that intertwines the canonical gauge actions (see \cite[Theorem~4.1]{Carlsen-Rout2017}), and \cite[Example~3.6]{BC20a} provides an example of two (finite) directed graphs (with no sinks and no sources) with boundary-path spaces that are eventually conjugate, but not conjugate.
\end{remark}

\vspace{4ex}
\end{document}